\documentclass[11pt, final]{amsart}
\usepackage[english]{babel}
\usepackage[latin1]{inputenc}
\usepackage{exscale}
\usepackage[centertags]{amsmath}
\usepackage{amsfonts,amstext,amssymb,amsthm, amsmath}
\usepackage[usenames,dvipsnames]{xcolor}
\usepackage[colorlinks=true,linkcolor=blue,citecolor=ForestGreen]{hyperref}
\usepackage{comment}
\usepackage{newlfont}
\usepackage{enumerate}
\usepackage{graphicx}
\usepackage{these}
\usepackage{color}
\usepackage{hyperref}
\usepackage{url}
\usepackage{geometry}
\usepackage{float}
\usepackage{multirow}% pacote para matrizes em blocos
\usepackage{graphicx}
\usepackage{amssymb}
\usepackage{verbatim} 
\usepackage{color}
\usepackage{epstopdf}
\usepackage{epsfig}
\usepackage[all]{xy}

\newtheorem{defi}{Definition}

\newtheorem{prop}{Proposition}
\newtheorem{lemma}{Lemma}

\newtheorem*{obs}{Remark}

\DeclareMathOperator{\perimetro}{\mathcal{P}}
\DeclareMathOperator{\HM}{\mathcal{H}}
\DeclareMathOperator{\spt}{spt}
\DeclareMathOperator{\supp}{supp}

\newcommand{\R}{\mathbb{R}}

\newcommand{\cluster}[1]{\mathcal{#1}}
\newcommand{\chamber}[2]{\cluster{#1}(#2)}
\newcommand{\intface}[3]{\cluster{#1}(#2, #3)}
\newcommand{\vol}[1]{\textbf{Vol}_{g}\left( #1\right)}
\newcommand{\vv}[1]{\mathbf{v}_{g}\left(\cluster{#1}\right)}
\newcommand{\diff}{\operatorname{d}\!}
\newcommand{\famdif}[1]{#1_{t}}
\newcommand{\per}[2]{\perimetro_{g}\left(#1, #2\right)}
\newcommand{\peri}[1]{\perimetro_{g}\left(#1\right)}

\newcommand{\difsim}[2]{#1 \Delta #2}
\newcommand{\infimo}[2]{\inf\left\{ #1 : #2 \right\}}

\newcommand{\norm}[1]{\Vert #1 \Vert}
\newcommand{\modulo}[1]{\vert #1 \vert}

\newcommand{\euc}[1]{\mathbb{R}^{#1}}
\newcommand{\hm}{\HM_{g}^{n-1}}
\newcommand{\hmde}[1]{\hm\left(#1\right)}
\newcommand{\geoball}[2]{\textbf{B}_{M}\!\left(#1, #2\right)}
\newcommand{\ball}[2]{\textbf{B}_{g}\!\left(#1, #2\right)}
\newcommand{\sphere}[1]{\mathbb{S}^{#1}}
\newcommand{\frored}[1]{\partial^{\ast}#1}
\newcommand{\cc}{\Subset}

\newcommand{\prodin}[2]{#1\cdot #2}

\vspace*{-0.5cm}

\title[\resizebox{4.5in}{!}{On Clusters and the multi-isoperimetric profile in Riemannian manifolds with bounded geometry}]{On Clusters and the multi-isoperimetric profile in Riemannian manifolds with bounded geometry}

\begin{document}
      \maketitle 
      \begin{center}
            \author{REINALDO RESENDE DE OLIVEIRA}\footnote{\emph{resende.reinaldo.oliveira@gmail.com - Instituto de Matem\'atica e Estat\'istica - Universidade de S\~ao Paulo - Brasil}}

\vspace{0.7cm}

\noindent {\sc abstract}. For a complete Riemannian manifold with bounded geometry, we prove the existence of isoperimetric clusters and also the compactness theorem for sequence of clusters in a larger space obtained by adding finitely many limit manifolds at infinity. Moreover, we show that isoperimetric clusters are bounded. We introduce and prove the H\"{o}lder continuity of the multi-isoperimetric profile which has been explored by Emanuel Milman and Joe Neeman with a  Gaussian-weighted notion of perimeter. We yield a proof of the classical existence theorem, e.g. in space forms, for isoperimetric cluster, using the results presented here. The results in this work generalize previous works of Stefano Nardulli, Andrea Mondino, Frank Morgan, Matteo Galli and Manuel Ritor\'e from the classical Riemannian and sub-Riemannian isoperimetric problem to the context of Riemannian isoperimetric clusters and also Frank Morgan and Francesco Maggi works on clusters theory in the Euclidean setting.
\bigskip

\bigskip
 
\centerline{\bf AMS subject classification: }
49Q20, 58E99, 53A10, 49Q05, 28A75.
\end{center}
      
      \tableofcontents   
      
      \newpage

\section{Introduction}\label{introduction}
In this paper, we study the existence of isoperimetric clusters for prescribed volume vector, i.e. minimizers of the perimeter functional under volume constraint, in a complete Riemannian manifold, assuming some bounded geometry conditions. The difficulty appears when the ambient manifold is noncompact since, in this case, it could happen that for a sequence of clusters with perimeter approaching the infimum, some volume may disappear at infinity and the limit of the sequence could not belong to the ambient manifold. We show that the cluster splits into a finite number of pieces (sub-clusters) that carry a positive fraction of the volume, one of them possibly staying at finite distance and the others concentrating along with divergent directions. Moreover, each of these pieces will converge to an isoperimetric cluster for its volume vector lying in some pointed limit manifold, possibly different from the original. So, isoperimetric clusters exist in this generalized sense as stated in Theorem \ref{Theorem:Existence}. The range of applications of these results is wide as it generalizes all the well known ideas carried by clusters in Euclidean spaces. We show that isoperimetric clusters are bounded in Proposition \ref{Prop:BoundednessClusters} by standard arguments used to prove the boundedness of clusters in Euclidean spaces and isoperimetric regions in Riemannian manifolds with bounded geometry. We also study properties of the multi-isoperimetric profile, defined in (\ref{multi-isoperimetric profile}), as its continuity stated in Theorem \ref{Theorem:HolderContinuity}. The theory that we will construct has its own importance and value, although we show how to apply it to prove a classical existence theorem of isoperimetric clusters in Theorem \ref{Theorem:ClassicalExistence}. The vague notions invoked in this introductory paragraph will be made clear and rigorous in the sequel.  

\subsection{Finite perimeter sets and its basic concepts}

In the remaining part of this paper we always assume that all the Riemannian manifolds $(M^n, g)$ considered are smooth with smooth Riemannian metric $g$ and $n\geq 2$. We denote by $\vol{}$ the canonical Riemannian measure induced on $M$ by $g$, and by $\hm$ the $(n-1)$-Hausdorff measure associated to the canonical Riemannian length space metric $d$ of $M$. When it is already clear from the context, explicit mention of the metric $g$ will be suppressed in what follows. First of all, we shall denote the metric ball centered at $p\in M$ of radius $r$ by $\ball{p}{r}$ and we define two concepts from Riemannian geometry.

\begin{defi} Let $(M^n, g)$ be a Riemannian manifold, the \textbf{\emph{injectivity radius}} of $M$, denoted $inj_M$, is defined as follow $$inj_M:=\inf_{p\in M}\{inj_{p,M}\},$$ where for every point $p\in M$, $inj_{p, M}$ is the injectivity radius at $p$ of $M$, i.e. the largest radius $r$ for which the exponential map $exp_p:B(0,r)\rightarrow \ball{p}{r}$ is a diffeomorphism.
\end{defi} 

The next definition places the concept of bounded geometry which will be crucial for all the theory developed in this work.

\begin{defi}\label{Def:BoundedGeometry}
A complete Riemannian manifold $(M^n, g)$, is said to have \textbf{\emph{bounded geometry}} if there exists a constant $k\in\mathbb{R}$, such that $Ric_g\geq k(n-1)$ (i.e., $Ric_g\geq k(n-1)g$ in the sense of quadratic forms) and $\vol{\geoball{p}{inj_M}}\geq v_0$ for some positive constant $v_0$, where $\geoball{p}{r}$ is the geodesic ball (or equivalently the metric ball) of $M$ centered at $p$ and of radius $r\in (0, inj_M)$.
\end{defi}

The reason to consider only manifolds with bounded geometry is explained in this introduction. Let us define the basic notions from the geometric measure theory.

\begin{defi} 
Let $(M^n,g)$ be a Riemannian manifold of dimension $n$, $U\subseteq M$ an open subset, $\mathfrak{X}_c(U)$ the set of smooth vector fields with compact support on $U$. Given a function $u\in L^{1}(M,g)$, define the {\bf \emph{total variation of $u$}} by

$$                 |Du|_{g}(M):=\sup\left\{\int_{M}u div_g(X)dv_g: X\in\mathfrak{X}_c(M), ||X||_{\infty ,g}\leq 1\right\}, $$

where $||X||_{\infty ,g}:=\sup\left\{|X_p|_{g}: p\in M\right\}$ and $|X_p|_{g}$ is the norm of the vector $X_p$ in the metric $g$ on $T_pM$.
      We say that a function $u\in L^{1}(M,g)$, has \textbf{\emph{bounded variation}}, if $|Du|_{g}(M)<+\infty$ and we define the set of all functions of bounded variations on $(M^n, g)$ by $BV(M,g):=\{u\in L^1(M,g):\:|Du|_{g}(M)<+\infty\}$.
\end{defi}

\begin{defi}
A function $u \in L^1_{loc}(M)$ has \textbf{\emph{locally bounded variation in $(M^n, g)$}}, if
for each open set $U \subset\subset M$, 
$$|Du|_{g}(U):=\sup\left\{ \int_U u div_g(X)dv_g : X\in\mathfrak X_c(U), \|X\|_{\infty,g}\le1  \right\}<+\infty,$$
and we define the set of all functions of locally bounded variations on $(M^n, g)$ by $BV_{loc}(M, g) := \{u \in L^1_{loc}(M, g) : |Du|_{g}(U) < +\infty, U\subset\subset M\}$. So for any $u\in BV(M,g)$, we can associate a vector Radon measure on $(M^n, g)$, denoted $Du$ with total variation $|Du|_g$.
\end{defi}

Given $E\subset M$ measurable with respect to the canonical Riemannian measure and $U\subset M$ an open subset, the \textbf{perimeter of $E$ in $U$}, $ \mathcal{P}_g(E, U)\in [0,+\infty]$, is

\begin{equation}\label{Def:Perimeter}
\mathcal{P}_g(E, U):= |D\chi_{E}|_{g}(U) = \sup\left\{\int_{U}\chi_E div_g(X)dv_g: X\in\mathfrak{X}_c(U), ||X||_{\infty ,g}\leq 1\right\}.
\end{equation} 
      
If $\mathcal{P}_g(E, U)<+\infty$ for every open set $U\subset\subset M$, we call $E$ a \textbf{locally finite perimeter set}. Let us set $\mathcal{P}_g(E):=\mathcal{P}_g(E, M)$. Finally, if $\mathcal{P}_g(E)<+\infty$ we say that \textbf{$E$ is a set of finite perimeter}.

In the case that the boundary of the set $E$ is not smooth, the topological boundary $\partial E$ of a set of finite perimeter is not a good candidate to measure the perimeter because its Hausdorff measure exceeds, in general, such value. The correct boundary in this context is the reduced boundary introduced by Ennio De Giorgi whose definition we recall below.

\begin{defi}  
The \textbf{\emph{reduced boundary}} of a set of finite perimeter $E$ of $(M^n, g)$, denoted by $\frored{E}$ is defined as the collection of points $p$ at which the
limit
$$
\nu_{E,g}(p) := \lim _{r\to 0} \frac{D^g \chi_{E}\left(\geoball{p}{r}\right)}{\left|D^g \chi_{E}\right|\left(\geoball{p}{r}\right)}
$$
exists and has length equal to one, i.e.
$$
\left|\nu_{E, g}(p)\right|_{g}=1.
$$
The function \(\nu_{E, g} : \partial^*E \rightarrow\mathbb{S}^{n-1}\) is called the \textbf{\emph{generalized inner normal to \(E\)}}.
\end{defi}

\begin{obs} By the very definition, $\frored{E}$ could depend on the metric $g$, but in fact it can be show that it does not depend on the choice of the metric $g$. To see this is a straightforward argument.
\end{obs}

By standard results of the theory of sets of finite perimeter, we have that $\peri{E, F} = \hmde{\frored{E}\cap F}$ where $\frored{E}$ is the reduced boundary of $E$ and $F\subset M$ is any Borel set. In particular, if $E$ has smooth boundary, then $\frored{E}=\partial E$, where $\partial E$ is the topological boundary of $E$. In the sequel, we will not distinguish between the topological boundary and reduced boundary when no confusion can arise.

\begin{defi}\label{Def:IsPWeak}
Let $(M^n, g)$ be a Riemannian manifold. We denote by $\tilde{\tau}_M$ the set of  finite perimeter subsets of $(M^n, g)$. The function $\tilde{I}_M:[0,\vol{M})\rightarrow [0,+\infty )$  defined by 
     $$\tilde{I}_M(v):= \inf\{\peri{\Omega}: \Omega\in \tilde{\tau}_M, \vol{\Omega}=v \}$$ 
is called the \textbf{\emph{isoperimetric profile function}} $($or shortly the \textbf{\emph{isoperimetric profile}}$)$ of the Riemannian manifold $(M^n, g)$. If there exist a finite perimeter set $\Omega\in\tilde{\tau}_M$ satisfying $\vol{\Omega}=v$, $\tilde{I}_M(\vol{\Omega})= \peri{\Omega}$ such an $\Omega$ will be called an \textbf{\emph{isoperimetric region}}, and we say that $\tilde{I}_M(v)$ is \textbf{\emph{achieved}}. 
\end{defi} 

For further results on the isoperimetric profile, one can consult \cite{pittet2000isoperimetric}, \cite{bayle2004differential}, \cite{nardulli2015discontinuous}, \cite{flores2014continuity}, etc.		

\begin{defi}\label{Def:L1Loc}
We say that a sequence of finite perimeter sets $E_j$ \textbf{\emph{locally converges to another finite perimeter set $E$}}, and we denote this by writing $E_j \stackrel{loc}{\rightarrow} E$, if $\chi_{E_j}\rightarrow\chi_{E}$ in $L^1_{loc}(M, g)$, i.e. if $\vol{\difsim{E_j}{E}\cap U}\rightarrow 0$, for all $U\subset\subset M$. Here $\chi_{E}$ means the characteristic function of the set $E$ and the notation $U\subset\subset M$ means that $U\subseteq M$ is open and $\overline{U}$ (the topological closure of $U$) is compact in $M$.
\end{defi}  

\begin{defi}\label{Def:ConvInTheSenseFPSets}
We say that a sequence of finite perimeter sets $E_j$ \textbf{\emph{converge in the sense of finite perimeter sets}} to another finite perimeter set $E$ if $E_j \stackrel{loc}{\rightarrow} E$ and 
\begin{eqnarray*} 
                           \lim_{j\rightarrow+\infty}\peri{E_j}=\peri{E}.
\end{eqnarray*}
\end{defi}

For a more detailed discussion on locally finite perimeter sets and functions of bounded variation, one can consult  \cite{MPPP} to the Riemannian setting and \cite{MorGMT}, \cite{MaggiBook2012}, \cite{Giu} and \cite{ambrosio2000functions} to the Euclidean setting.

\subsection{On $C^{m, \alpha}$-convergence of manifolds and bounded geometry}

Now, let us recall the basic definitions from the theory of convergence of manifolds, as exposed in \cite{Pet}. This will help us to state the main result in a precise way. 

\begin{defi}
For any $m\in\mathbb{N}$, $\alpha\in [0, 1]$, a sequence of pointed smooth complete Riemannian manifolds is said to \textbf{\emph{converge in
the pointed $C^{m,\alpha}$, respectively $C^{m}$, topology to a smooth manifold $M$}} (denoted by $(M_i, g_i, p_i)\rightarrow (M,g,p)$), if for every $R > 0$ we can find a domain $\Omega_R$ with $B(p,R)\subseteq\Omega_R\subseteq M$, a natural number $\nu_R\in\mathbb{N}$, and $C^{m+1}$ embeddings $F_{i,R}:\Omega_R\rightarrow M_i$, for large $i\geq\nu_R$ such that $F_{i, R}(p) = p_i, B(p_i,R)\subseteq F_{i,R} (\Omega_R)$ and $F_{i,R}^*(g_i)\rightarrow g$ on $\Omega_R$ in the $C^{m,\alpha}$, respectively $C^m$ topology. 
\end{defi}

\begin{defi}
We say that a sequence of multipointed Riemannian manifolds $(M_i, g_i, p_{1i}, \cdots, p_{ji}, \cdots)$ \textbf{\emph{converges to the multipointed Riemannian manifold $(M, g, p_{1}, \cdots, p_{j}, \cdots)$ in the multipointed $C^0$-topology}}, if for every $j$ we have

$$(M_i, g_i, p_{ji}) \to (M, g,  p_{j})  $$ 

in the pointed $C^0$-topology.
\end{defi}

It is easy to see that this type of convergence implies pointed Gromov-Hausdorff convergence. We now define the notion of norm of sets in a manifold which is the \textbf{$C^{m,\alpha}$-norm at scale $r$}. This notion can be taken as a possible definition of bounded geometry. 

\begin{defi}[\cite{Pet}]\label{Def:BoundedGeometryPetersen}
A subset $A$ of a Riemannian n-manifold $M$  has \textbf{\emph{bounded $C^{m,\alpha}$ norm on the scale of $r$}}, $||A||_{C^{m,\alpha},r}\leq Q$, if every point p of M lies in an open set $U$ with a chart $\psi$ from the Euclidean $r$-ball into $U$ such that
\begin{enumerate}[(i):]
      \item $|D\psi|\leq e^Q$ on $B(0,r)$ and $|D\psi^{-1}|\leq e^Q$ on $U$,
      \item $r^{|j|+\alpha}||D^jg||_{\alpha}\leq Q$ for all multi indices j with $0\leq |j|\leq m$, where $g$ is the matrix of functions of metric coefficients in the $\psi$ coordinates regarded as a matrix on $B(0,r)$. 
\end{enumerate} 
We define the set $\mathcal{M}^{m,\alpha}(n, Q, r)$ as the set of pointed manifolds $(M,g,p)$ which satisfies $||M||_{C^{m,\alpha},r}\leq Q$. 
\end{defi}

In the sequel, unless otherwise specified, we will make use of the technical assumption on $(M,g,p)\in\mathcal{M}^{m,\alpha}(n, Q, r)$ that $n\geq 2$, $r,Q>0$, $m\geq 1$, $\alpha\in ]0,1]$.  Roughly speaking, $r>0$ is a positive lower bound on the injectivity radius of $M$, i.e. $inj_M>C(n,Q,\alpha,r)$. 

In general, a lower bound on $Ric_M$ and on the volume of unit balls, i.e. the bounded geometry requirements (Definition \ref{Def:BoundedGeometry}), does not ensure that the pointed limit metric spaces at infinity are still manifolds, this motivates the following definition.

\begin{defi}\label{Def:BoundedGeometryInfinity}
We say that a smooth Riemannian manifold $(M^n, g)$ has \textbf{\emph{$C^{m,\alpha}$-bounded geometry}} if it is of bounded geometry and if for every diverging sequence of points $(p_j)$, there exist a subsequence $(p_{{j}_{l}})$ and a pointed smooth manifold $(M_{\infty}, g_{\infty}, p_{\infty})$ with $g_{\infty}$ of class $C^{m,\alpha}$ such that the sequence of pointed manifolds $(M, g, p_{{j}_{l}})\rightarrow (M_{\infty}, g_{\infty},  p_{\infty})$, in  $C^{m,\alpha}$-topology.  
\end{defi}

We observe here that Definition \ref{Def:BoundedGeometryInfinity} is weaker than Definition \ref{Def:BoundedGeometryPetersen}. In fact, using Theorem $72$ of \cite{Pet}, one can show that if a manifold $M$ has bounded $C^{m,\alpha}$ norm on the scale of $r$ for $\alpha>0$ in the sense of Definition \ref{Def:BoundedGeometryPetersen} then $M$ has $C^{m,\alpha}$-bounded geometry in the sense of Definition \ref{Def:BoundedGeometryInfinity}, while in general the converse is not true.  

In the absence of the extra condition of Definition \ref{Def:BoundedGeometryInfinity}, just assuming bounded geometry in the sense of Definition \ref{Def:BoundedGeometry}, the resulting limit space is merely a length space $(Y,d_Y,y)$. For more on the structure of these limit spaces, one can consult for example the works of Cheeger-Colding (\cite{ChCold1}, \cite{ChCold2}, \cite{ChCold3}). 
Regarding the smooth structure of limit spaces $(Y, d_Y, y)$, the reader is referred to Cheeger-Anderson \cite{AndChe92}, Anderson \cite{An92}, or \cite{Pet} chapter 10 for a more expository discussion.

\subsection{Clusters}

The definition of clusters that we adopt is motivated by the definition given by Francesco Maggi in the Euclidean setting in \cite{MaggiBook2012}. For other references in cluster theory in the Euclidean setting, one can consult \cite{MorGMT} and \cite{Almgren76Memoirs}.

\begin{defi}\label{Def:Clusters} Let $(M^n,g)$ be a Riemannian manifold of dimension $n$. An \textbf{\emph{N-cluster}} $\mathcal{E}$ of $(M^n, g)$ is a finite family of sets of finite perimeter $\mathcal{E}:=\left\{ \mathcal{E}(h)\right\}^N_{h=1}$, $N \in \mathbb{N}$, $N \geq 1$, with 
\[
%\begin{cases}
0 < \vol{\mathcal{E}(h)}< +\infty, \hspace{0.9cm} 1 \leq h \leq N,\]
\[
\vol{\mathcal{E}(h) \cap \mathcal{E}(k)} = 0, \hspace{0.7cm} 1 \leq h < k \leq N.
%\end{cases}
\]
The sets $\mathcal{E}(h)$ are called the {\bf \emph{chambers}} of $\mathcal{E}$. When the number $N$ of the chambers of $\mathcal{E}$ is clear from the context, we shall use the term "cluster" in place of "N-cluster". The {\bf \emph{exterior chamber}} of $\mathcal{E}$ is defined as 

$$ \mathcal{E}(0) = M^n\setminus\bigcup^N_{h=1} \mathcal{E}(h). $$

In, particular, $\left\{ \mathcal{E}(h)\right\}^N_{h=0}$ is a partition of $M^n$ (up to a set of null volume). The \textbf{\emph{volume vector}} $\vv{E}$ is defined as 
\[
\vv{E}= \left( \vol{\chamber{E}{1}}, \ldots , \vol{\chamber{E}{N}} \right) \in\mathbb{R}^N. 
\]

We let $\mathbb{R}^N_{+}$ be the set of those $\v\in \mathbb{R}^N$ such that $\v(h) > 0$ (the $h$-th component of a vector $\v$) for every $h = 1, \ldots , N.$ Notice that if $\mathcal{E}$ is an N-cluster, then $\vv{E} \in \mathbb{R}^N_{+}=[0,\vol{M}[^N$ as $\vv{E}(h) = \vol{\chamber{E}{h}} > 0$ for every $h=1, \ldots,N.$
\end{defi}

\begin{obs} It is important to notice that the chambers of a cluster are not assumed to be indecomposable, it is known that indecomposability is the commonly accepted notion of connectedness in the framework of sets of finite perimeter.
\end{obs} 

\begin{defi}
The {\bf \emph{interfaces}} of the N-cluster $\mathcal{E}$ in $(M^n, g)$ are the $\hm$-rectifiable sets
\[\mathcal{E}(h,k) = \frored{\chamber{E}{h}}\cap \frored{\chamber{E}{k}}, \hspace{0.4cm} 0\leq h,k \leq N, h \neq k. \]

We define the {\bf \emph{relative perimeter of $\cluster{E}$ in $F \subset M^n$}} as 
\begin{equation}\label{relative-perimeter}
\per{\cluster{E}}{F} = \sum_{0\leq h<k\leq N}\hmde{F\cap\intface{E}{h}{k})},
\end{equation}
where $F$ is any Borelian set in $(M^n, g)$. The {\bf \emph{perimeter of $\cluster{E}$}} is denoted $\perimetro(\cluster{E}) \doteq \per{\cluster{E}}{M}$.
\end{defi}

\begin{defi}
The \textbf{\emph{flat distance or flat norm in $F \subset M^n$ of two N-clusters $\cluster{E}$ and $\cluster{E'}$ of $(M^n, g)$}} is defined as
\[ d_{\mathcal{F}, g}^F(\cluster{E},\cluster{E'}):=\sum^N_{h=1} \vol{F \cap (\difsim{\chamber{E}{h}}{\chamber{E'}{h}})}.\]

We set $d_{\mathcal{F}, g}(\mathcal{E},\mathcal{E}') = d_{\mathcal{F}, g}^{M^n}(\mathcal{E},\mathcal{E}')$. With this notation at hand, we say that a sequence of N-clusters $\left\{ \cluster{E}_k \right\}_{k \in \mathbb{N}}$ in $(M^n, g)$ {\bf \emph{locally converges to $\cluster{E}$}}, and write $\cluster{E}_k \stackrel{loc}{\rightarrow} \cluster{E}$, if for every compact set $K \subset M^n$ we have $d_{\mathcal{F}, g}^K(\cluster{E},\cluster{E}_k)\to 0$ as $k \rightarrow +\infty$. If $d_{\mathcal{F}, g}(\cluster{E}, \cluster{E}_k) \to 0$ as $k\to+\infty$, we say that $\cluster{E}_k$ {\bf \emph{converges to $\cluster{E}$}} and we denote $\cluster{E}_k \to \cluster{E}$.
\end{defi}

In order to simplify our formula for the relative perimeter, we will prove the following result which  interestingly makes easy to prove the lower semicontinuity for sequences of clusters. In (\ref{relative-perimeter}), we have the problem of working with the interfaces of the clusters which can be a tough task, since the intersection of the reduced boundary is not the same set of the reduced boundary of the intersection. To avoid this kind of problem, the formula provided by the following proposition turns out to be one of the key ideas to work with clusters in the way that it permits us to work with the perimeter of each chamber \emph{separately} in a sum, as shown in the following proposition which is the Riemannian counterpart of Proposition 29.4 of \cite{MaggiBook2012}.

\begin{prop}\label{EquationToPerimeter} If $\mathcal{E}$ is an N-cluster in $(M^n, g)$, then for every $F \subset M^n$ we have
\begin{equation}\label{perimeter-formula}
\P_g(\mathcal{E};F) = \frac{1}{2} \sum_{h=0}^{N}\P_g(\mathcal{E}(h);F).
\end{equation}
In particular, if A is open in $M^n$ and $\mathcal{E}_k \stackrel{loc}{\rightarrow} \mathcal{E},$ then
\begin{equation}\label{lsc}
\P_g(\mathcal{E};A) \leq \liminf_{k \to +\infty}\P_g(\mathcal{E}_k;A).
\end{equation}
\end{prop}

For completeness, let us state the classical compactness criterion for a sequence of clusters which are contained in some fixed ball. The assumptions of the sequence of clusters be subset of a compact manifold (the ball they are contained in) is crucial for this compactness result, since we can use a simple tool that is the compactness criterion for {\bf BV} functions which prove the criterion almost automatically. Without this assumptions, the problem of showing the existence of a "limit cluster" turns out to be quite hard and in general for a complete Riemannian manifold this result is no longer true, because, as it is well known, some part with positive volume of the sequence of clusters could disappear at infinity.

\begin{prop}[Compactness criterion for clusters]\label{Prop:ClassicalClusterCompactness} If $\left\{ \mathcal{E}_k \right\}_{k \in \mathbb{N}}$ is a sequence of N-clusters in $(M^n, g)$,

$$ \sup_{k \in \mathbb{N}}\peri{\mathcal{E}_k} < +\infty,$$
$$\inf_{k \in \mathbb{N}} \min_{1 \leq h \leq N} \vol{\mathcal{E}_k(h)} > 0$$
and
$$ \mathcal{E}_k(h) \subset \ball{p}{R}, \hspace{0.2cm} \forall k \in \mathbb{N}, h = 1, \ldots, N,$$

$R>0,$ for some $p\in M$, then there exist an N-cluster $\mathcal{E}$ in $(M^n, g)$ with $\chamber{E}{h}\subset \ball{p}{R}$ such that up to a subsequence $\cluster{E}_k\to\cluster{E}$ as $k \longrightarrow +\infty.$
\end{prop}
\begin{proof}
The proof goes along the same lines of the Euclidean proof, see Proposition 29.5 in \cite{MaggiBook2012}.
\end{proof}

\begin{obs}[Density properties at interfaces] If $\mathcal{E}$ is a N-cluster in $(M^n, g)$, $p \in \mathcal{E}(h,k)$, $0 \leq h < k \leq N, j\neq h, k,$ then we get

\begin{equation}\label{normalinterfaces}
\nu_{\mathcal{E}(h)}(p) = - \nu_{\mathcal{E}(k)}(p),
\end{equation}
\begin{equation}\label{density0}
\theta_n (\mathcal{E}(j))(p) = 0,
\end{equation}
\begin{equation}\label{density0'}
\theta_{n-1}(\partial^{\ast}\mathcal{E}(j))(p) = 0.
\end{equation}

\end{obs}

We set the following notations that were used in the previous remark:

\[ \theta_n(E)(p) = \lim_{r\to 0}\frac{\vol{E\cap\ball{p}{r}}}{\omega_n r^n}\]
and
\[ \theta_{n-1}(E)(p) = \lim_{r\to 0}\frac{\per{E}{\ball{p}{r}}}{\omega_{n-1} r^{n-1}}\]

for any finite perimeter set $E\subset M$ and $x\in M$, where $\omega_n$ is the volume of the unit ball on the Euclidean space of dimension $n$.

\begin{obs} If $\mathcal{E}$ is an N-cluster and $\Lambda \subset \left\{ 0, \ldots, N \right\},$ then 
\[
\hmde{\partial^{\ast} \left( \bigcup _{h \in \Lambda} \mathcal{E}(h) \right)\setminus\bigcup_{h \in \Lambda, k \notin \Lambda} \mathcal{E}(h,k)} = 0.
\]
\end{obs}

\subsection{Isoperimetric clusters}

The main goal of this paper is to prove the existence of isoperimetric clusters in a generalized sense. An {\bf isoperimetric cluster for volume $\v\in\mathbb{R}^{N}_{+}$} is an N-cluster $\cluster{E}$ that solves the minimizing problem below which is also known as {\bf multi-isoperimetric problem}, i.e. such that $\vv{E} = \v$ and

$$ \peri{\cluster{E}} = \infimo{\perimetro(\cluster{E'})}{\cluster{E'} \ \text{is an N-cluster with} \ \vv{\cluster{E'}} = \mathbf{v}}.$$

Similarly to the {\bf isoperimetric problem} context, i.e. $N=1$, we can define the {\bf multi-isoperimetric profile function}, or {\bf multi-isoperimetric profile}, as a function $I_M$ from $[0, \vol{M})^N$ to $[0, +\infty )$ given by

\begin{equation}\label{multi-isoperimetric profile}
I_M(\v) = \infimo{\perimetro(\cluster{E})}{\cluster{E} \ \text{is an N-cluster in} \ (M^n,g) \ \text{with} \ \vv{\cluster{E}} = \mathbf{v}}.
\end{equation}

This generalization of the isoperimetric profile for clusters has been explored in \cite{milman2018gaussian} with a Gaussian-weighted notion of perimeter which is induced by the Gaussian probability measure. These notions are all used in this work to solve and investigate the Gaussian Multi-Bubble Conjecture.

We notice that we prove the H\"{o}lder continuity of the multi-isoperimetric profile in Theorem \ref{Theorem:HolderContinuity}. If either $M$ is compact or exists a minimizing sequence contained in a compact subset of $M$, classical compactness arguments of geometric measure theory, i.e. Proposition \ref{Prop:ClassicalClusterCompactness} , combined with the direct method of the calculus of variations provide existence of isoperimetric clusters in any dimension $n$. These arguments are derivations from the theory of clusters in Euclidean spaces, for this setting we refer the reader to \cite{ColomboMaggi2017}, \cite{Mor94}, \cite{hirsch2020lower} and \cite{MaggiBook2012}.

In the case that $N=1$, we return to the classical isoperimetric problem which was extensively studied. The existence of isoperimetric regions in noncompact Riemannian manifolds is not a easy task. However, we can find papers in this directions which give pretty good answers in some specific types of manifolds. For an example, in \cite{RitGalli}, Ritor\'e and Galli proved the existence of isoperimetric regions to the case of noncompact sub-Riemannian manifolds with cocompact isometry group. For the Riemannian setting, we refer the reader to \cite{Morg1}, \cite{Ritsec}, \cite{RRosales}, \cite{RitCan}, \cite{Nar}, \cite{NarAnn}, \cite{flores2019local} and  \cite{nardulli2014generalized}. For more details on regularity theory see either \cite{Morg1} or \cite{MorGMT}. Accordingly to these references, we could extract that we need some condition on the geometry of the manifold to prove existence of isoperimetric regions which we will call {\bf bounded geometry}, defined in Definition \ref{Def:BoundedGeometry}. This condition has been studied by several mathematicians and \cite{nardulli2015discontinuous} provided a counter example of a manifold which does not satisfy one of the bounded geometry conditions and hence does not contain isoperimetric regions for some volumes. 

\subsection{Main theorems}

With the notions of multipointed $C^0$-convergence and basic clusters concepts, we can enunciate the generalized compactness theorem which assumes that the sequence of clusters has uniformly bounded perimeter and components of the volume vectors and then ensures the existence of a limit cluster in the multipointed $C^0$-topology at a union (possibly infinite) of multipointed limit manifolds.

\begin{Res}[Generalized Compactness for Sequences of Clusters]\label{Theorem:Compactness} Suppose that $(M^n,g)$ has $C^0$-bounded goemetry. Let $\{ \cluster{E}_k \}_{k\in\mathbb{N}}$ be a sequence of $N$-clusters in $(M^n, g)$ with $\mathcal{P}_g(\cluster{E}_k) \leq P$ and $\vv{E_\mathit{k}}(h) \leq \v(h)$, for $h\in\{1, ..., N\}$. Then, up to a subsequence, there exists $J\in\mathbb{N}\cup\{+\infty\}$ such that, for all $j\in\{1, ... J\}$, there exist a sequence of points $(p_{jk}^h)_{k\in\mathbb{N}}\subset M$,  a manifold $(M_{\infty}(h), g)$, $(p_{j\infty}^h)_{k\in\mathbb{N}}\subset M_{\infty}(h)$ and a finite perimeter set $\cluster{E}_{\infty}(h)\subset M_{\infty}(h)$, $1\leq h \leq N$, such that 

$$ (\cluster{E}_k(h), g, p_{jk}^h) \ \text{converges to} \ (\cluster{E}_{\infty}(h), g, p^h_{j\infty})  $$

in the multipointed $C^0$-topology. Moreover, if we define the $N$-cluster $\cluster{E}_{\infty} = \{ \cluster{E}_{\infty}(h) \}_{h=1}^{N}$ in the manifold $M\cup\left(\cup_{h=1}^{N}M_{\infty}(h)\right)$, then $\vv{E_{\infty}} = \lim_{k\to+\infty}\vv{E_\mathit{k}}$ and $\peri{\cluster{E}_{\infty}} = \lim_{k\to+\infty}\peri{\cluster{E}_k}$.
\end{Res}

As discussed in the introductory section, the main theorem of this work is the generalized existence of isoperimetric clusters, i.e. a cluster $\cluster{E}$ that possibly lives at a union of limit multipointed manifolds which satisfies $I_M(\v) = \peri{\cluster{E}}$, it is stated rigorously below.

\begin{Res}[Generalized Existence of Isoperimetric Clusters]\label{Theorem:Existence} Suppose that $(M^n,g)$ has $C^0$-bounded goemetry. Let $\{ \cluster{E}_k \}_{k\in\mathbb{N}}$ be a minimizing sequence of $N$-clusters for $\v\in\mathbb{R}^{N}_+$. Then, up to a subsequence, there exists $J\in\mathbb{N}$, a manifold $(M_{\infty}, g)$, $J$ sequences of points $(p_{jk}^h)_{k\in\mathbb{N}}\subset M$, $(p_{j\infty}^h)_{k\in\mathbb{N}}\subset M_{\infty}$ and a $N$-cluster $\cluster{E}_{\infty}$ in $(M_{\infty}, g)$ such that 

$$ (\cluster{E}_k(h), g, p_{jk}^h) \ \text{converges to} \ (\cluster{E}_{\infty}(h), g, p^h_{j\infty}),  $$

for $h\in\{1, ..., N\}$, in the multipointed $C^0$-topology. Moreover, $\vv{E_{\infty}} = \v$ and $\peri{\cluster{E}_{\infty}} = I_{M_{\infty}}(\v) = I_{M}(\v) $.
\end{Res}

We also show in the next result that isoperimetric clusters are always bounded which is the analogous of Theorem 3 of \cite{nardulli2014generalized} to the context of clusters.

\begin{Res}[Boundedness of Isoperimetric Clusters]\label{Prop:BoundednessClusters}
Let $(M^n, g)$ be a Riemannian manifold with bounded geometry, then isoperimetric clusters are bounded.
\end{Res}

Since the results that we have been working with are supposed to generalize the classical existence and compactness results, it is natural to provide a proof of the classical result statement applying the previous results of this paper.

\begin{Res}[Classical Existence of Isoperimetric Clusters]\label{Theorem:ClassicalExistence}
Let $(M^n, g)$ be $C^0$-locally asymptotically the $n$-dimensional space form $\mathbb{M}_{k}^n$ of curvature $k$, $Ric_g \geq k(n-1)$ $($i.e. $Ric_g \geq k(n-1)g$ in the sense of quadratic forms$)$. Then, for every $\v\in\mathbb{R}^{N}_{+}$, there exist an isoperimetric cluster, i.e. an $N$-cluster $\cluster{E}$ with 

$$ I_M(\v) = \peri{\cluster{E}}. $$
\end{Res}

Now, we state the H{\"o}lder continuity of the multi-isoperimetric profile.

\begin{Res}[Local H{\"o}lder continuity of the multi-isoperimetric profile]\label{Theorem:HolderContinuity}
Let $(M^n, g)$ be a manifold with bounded geometry. Then there exists a constant $C(n, k) > 0$ such that for every $\v, \v' \in ]0, \vol{M}[^{N}$ satisfying $\v' \in\textbf{B}_{\mathbb{R}^{N}}\left( \v , R_\v\right)$, where 

$$ R_{\v} = \frac{1}{C(n, k)}\min\biggl \{v_0, \sum_{h=1}^{N}\biggl( \frac{\v(h)}{I_M(\v) + C(n, k)} \biggr)^n \biggr \},$$

we have that 

$$ \modulo{ I_M(\v) - I_M(\v')} \leq C(n, k) \biggl( \frac{\modulo{\v - \v'}}{v_0} \biggr)^{\frac{n-1}{n}}. $$

\end{Res}

\subsection{Plan of the article}
\begin{enumerate}
           \item  Section \ref{introduction} constitutes the introduction of the basic concepts, the contextualization of the problem, a basic notion of manifold's convergence theory and the statements of the main results.
           \item In Section \ref{proof-main-theorem}, we prove the main theorem (Theorem \ref{Theorem:Existence}), the generalized compactness theorem (Theorem \ref{Theorem:Compactness}) and also the H{\"o}lder continuity of the multi-isoperimetric profile (Theorem \ref{Theorem:HolderContinuity}). Moreover,  we verify that isoperimetric clusters are bounded (Theorem \ref{Prop:BoundednessClusters}).
           \item In Section \ref{application-classical-results}, we show how the generalized existence theorem (Theorem \ref{Theorem:Existence}) is used to prove classical existence theorems for isoperimetric clusters (Theorem \ref{Theorem:ClassicalExistence}).
           \item In the Appendix (Section \ref{appendix}), we furnish theorems which show how exchanging volume process and volume fixing variations work in the Riemannian setting.
\end{enumerate}

\subsection{Acknowledgements}  
This article is part of my Ph.D thesis written under the advising of Stefano Nardulli. I would like to give a special thanks to Stefano Nardulli, for his enthusiasm for the project, for his contribution with precious ideas and for bringing my attention to the subject of this paper. The discussions, encouragements and comments of my co-advisor, Glaucio Terra, were very valuable for this work. I also show appreciation to Frank Morgan for his edits of the original text. This study was financed in part by the Coordena\c{c}\~{a}o de Aperfei\c{c}oamento de Pessoal de N\'{i}vel Superior - Brasil (CAPES) - Finance Code 88882.377954/2019-01. 

\newpage

\section{Proof of the main theorem}\label{proof-main-theorem}

\subsection{The structure lemma and the generalized compactness theorem}

We will state the Lemma that provides a kind of structure for sequence of clusters with uniformly bounded perimeter and components of the volume vector. It will be useful since it splits each chamber of the cluster in various pieces (which can be infinite) with properties of volume and perimeter being well preserved. Moreover, some convergence properties in the $C^{m, \alpha}$ or Gromov-Hausdorff convergence sense will work very well too. 

\begin{lemma}[Structure Lemma for Sequences of Clusters]\label{Lemma:StructureLemma} Assume that $(M^n, g)$ has bounded geometry. Let $\{ \cluster{E}_k \}_{k\in\mathbb{N}}$ be a sequence of $N$-clusters in $(M^n, g)$ with $\mathcal{P}_g(\cluster{E}_k) \leq P$ and $\vv{E_\mathit{k}}(h) \leq \v(h)$, for $h\in\{1, ..., N\}$, where $P\in \mathbb{R}_{+}$ and $\v\in\mathbb{R}^{N}_{+}$. Then there exists $J\in\mathbb{N}\cup\{+\infty\}$ such that, for all $j\in\{1, ... ,J\}$, there exist a sequence of points $(p_{jk}^h)_{k\in\mathbb{N}}$, a sequence of radii $R_{jk}\to +\infty$ as $k$ goes to $+\infty$ and a volume vector $\v_j\in\mathbb{R}^{N}$, such that, if we set $\cluster{E}^{0}_k(h) =  \cluster{E}_k(h)\cap\ball{p_{0k}^h}{R_{0k}}$,

$$\cluster{E}^j_k(h) = \cluster{E}_k(h)\cap\ball{p_{jk}^h}{R_{jk}}\setminus\mathring{\bigcup}_{i=0}^{j-1}\cluster{E}_k^{i}(h),  $$

for $j\geq 1$, and the $N$-cluster $\cluster{E}'_k = \{\cup_{j=0}^J\cluster{E}_k^j(h)\}_{h=1}^N$, the following properties hold

\begin{enumerate}[$(i):$]

	\item $\v_j(h) = \lim_{k\to+\infty}\vol{\cluster{E}^j_k(h)}$, for each $h\in\{1, ..., N\}$,
		
	\item $0<\vv{E'_{\mathit{k}}}(h) = \sum_{j=0}^J \v_j(h) \leq \v(h)$. Moreover, if $\vv{E_\mathit{k}}\to \v$, then $\vv{E'_\mathit{k}} \to \sum_{j=0}^J \v_j = \v,$
	
	\item $\lim_{k\to+\infty}\mathcal{P}_g(\cluster{E}'_k)\leq \lim_{k\to+\infty}\mathcal{P}_g(\cluster{E}_k) \leq P$.
	
\end{enumerate}

\end{lemma}
\begin{proof}
By the Concentration-Compactness Lemma (Lemma 2.1 of \cite{nardulli2014generalized}, the bounded geometry hypothesis, Definition \ref{Def:BoundedGeometry}, guarantees that property $(iii)$ of the Lemma 2.1 holds), there exist $\v_0(h)\in (0, \v(h))$ and a sequence of points $(p_{0k}^h)_{k\in\mathbb{N}}$ such that for any sequence $\epsilon^{0,k} \to 0$ there exists a sequence of radii $(R^{0,k})_{k\in\mathbb{N}}$ which satisfies 

$$ \biggl | \vol{\cluster{E}_k(h) \cap \ball{p_{0k}^h}{R'}} - \v_0(h) \biggr | < \epsilon^{0,k}, $$

whenever $R' \geq R^{0,k}$ and $k$ sufficiently large. By the mean value property for integrals, we find a radii $R_{0k}\in [R^{0, k}, R^{0, k}+k]$ such that

\begin{equation}\label{StructureLemmaEq1}
\peri{\cluster{E}_k(h)\cap\ball{p_{0k}^h}{R_{0k}}} = \frac{1}{R^{0, k} + k - R^{0, k}}\int_{R^{0, k}}^{R^{0, k}+k}\peri{\cluster{E}_k(h)\cap\ball{p_{0k}^h}{r}}\diff r.
\end{equation} 

Without loss of generality, we can assume that $\hmde{\frored{\cluster{E}_k(h)}\cap\partial\ball{p_{0k}^h}{R_{0k}}}=0$, therefore, this assumption, the operations with the perimeter measure, (\ref{StructureLemmaEq1}) and the coarea formula into account to get that

\begin{equation}\label{StructureLemmaEq2}
\begin{split}
\sum_{h=1}^{N}\peri{\cluster{E}_k(h)\cap\ball{p_{0k}^h}{R_{0k}}} = \\
 \sum_{h=1}^{N}\frac{1}{k}\int_{R^{0, k}}^{R^{0, k}+k}\per{\cluster{E}_k(h)}{\ball{p_{0k}^h}{r}^{(1)}} + \per{\ball{p_{0k}^h}{r}}{\cluster{E}_k(h)^{(1)}}\diff r \leq \\
\sum_{h=1}^{N}\peri{\cluster{E}_k(h)} + \frac{1}{k}\sum_{h=1}^{N}\v(h) \leq P + \frac{1}{k}\sum_{h=1}^{N}\v(h).
\end{split}
\end{equation} 

We will fix the notation $\cluster{E}_k^0(h) \doteq \cluster{E}_k(h) \cap \ball{p_{0k}^h}{R_{0k}}$. In order to repeat this process and therefore obtain a better approach for the volume of the chamber $\cluster{E}_{k}(h)$, we apply the Concentration-Compactness Lemma for the set

$$\cluster{E}_k(h)\setminus\cluster{E}_k^0(h).$$

Thus, we get the existence of $\v_1(h)\in (0, \v_1(h))$ and a sequence of points $(p_{1k}^h)$ such that for all sequence $\epsilon^{1,k} \to 0$ there exists a sequence of radii $(R^{1,k})_{k\in\mathbb{N}}$ with the following property

$$ \biggl | \vol{\cluster{E}_k(h) \cap \ball{p_{1k}^h}{R'}\setminus\cluster{E}_k^0(h)} - \v_1(h) \biggr | < \epsilon^{1,k}, $$

whenever $R' \geq R^{1,k}$ and $ k$ sufficiently large and then we also set $\cluster{E}_k^1(h) = \cluster{E}_k(h) \cap \ball{p_{1k}^h}{R_{1k}}\setminus\cluster{E}_k^0(h)$ with $R_{1k}$ such that (\ref{StructureLemmaEq2}) holds for $\cluster{E}_k^1(h)$. Now, we are in position to apply the Concetration-Compactness Lemma inductively, using in the j-th step, $j\geq 2$, the set

$$ \cluster{E}_k(h)\setminus\bigcup_{i=0}^{j-1}\cluster{E}_k^{i}(h),$$

hence obtaining $\v_j(h)\in (0, \v_{j-1}(h))$ and $(p_{jk}^h)$ such that for all sequence $\epsilon^{j,k} \to 0$ there exists a sequence of radii $(R^{j,k})_{k\in\mathbb{N}}$ with the following property

\begin{equation}\label{structurelemma1}
\biggl | \vol{\cluster{E}_k(h)\cap\ball{p_{jk}^h}{R'}\setminus\bigcup_{i=0}^{j-1}\cluster{E}_k^{i}(h)} - \v_j(h) \biggr | < \epsilon^{j,k},
\end{equation}  

whenever $R' \geq R^{j,k}$ and $k$ sufficiently large, we inductively denote by 

$$\cluster{E}_k^j(h) =  \cluster{E}_k(h)\cap\ball{p_{jk}^h}{R_{jk}}\setminus\bigcup_{i=0}^{j-1}\cluster{E}_k^{i}(h), $$

where $R_{jk}$ is taken to satisfy (\ref{StructureLemmaEq2}) holds for $\cluster{E}_k^j(h)$. So, we shall iterate the algorithm until we reach the desired $J\in\mathbb{N}\cup\{+\infty\}$. Finally, we certainly have that

$$ \v_j(h) = \lim_{k\to+\infty}\vol{\cluster{E}_\mathit{k}^\mathit{j}(h)} \leq \v(h). $$

We can easily see by the construction that the first assertion of item $(ii)$ is already proved. Therefore we suppose that $\vv{E_\mathit{k}} \to \v$ and $\sum_{j=0}^J \v_j(h) < \v(h)$, it is a direct consequence of Lemma 2.5 of \cite{nardulli2014generalized} applied to the set $\cluster{E}_k(h)\setminus\cluster{E}_k^j(h)$ that exists $p'\in M$ such that

$$ \v_{j+1}(h) \geq \vol{\ball{p'}{R^{j, k}}\cap\biggl(\cluster{E}_k(h)\setminus\cluster{E}_k^j(h)\biggr)} \geq c_3(n, k, v_0)\frac{\vol{\cluster{E}_k(h)\setminus\cluster{E}_k^j(h)}^n}{\mathcal{P}_g(\cluster{E}_k(h)\setminus\cluster{E}_k^j(h))^n}.$$ 

Passing through the limit as $k$ and $j$ goes to $+\infty$, we obtain that

$$ \lim_{j\to+\infty}\v_j(h) \geq c_3(n, k, v_0)\frac{\biggl( \v(h) - \sum_{j=1}^J \v_j(h) \biggr)^n}{\mathcal{P}_g(\cluster{E}_k(h)\setminus\cluster{E}_k^j(h))^n} > 0. $$

which is a contradiction with the fact that $\sum_{j=0}^J \v_j(h)$ is a convergent series and thus ensuring the validity of $(ii)$. Since (\ref{StructureLemmaEq2}) is true for all $j\in\{0, 1,..., J\}$, Proposition \ref{EquationToPerimeter} finishes the proof of $(iii)$ and of this Lemma.
\end{proof}

Let us prove the generalized compactness for sequences of clusters (Theorem \ref{Theorem:Compactness}).

\begin{proof}[Proof of Theorem \ref{Theorem:Compactness}]
By Gromov's Compactness Theorem and a diagonalization process applied to the sets $\cluster{E}^j_k(h)$ of Lemma \ref{Lemma:StructureLemma}, we ensure the existence of the manifolds $(M_{\infty}(h), g)$, the sequence of points $(p_{j\infty}^h)$ and the finite perimeter set $\cluster{E}_{\infty}(h)$. Theorem 4.0.6 of \cite{dai2012comparison} shows that the volume vectors will converge as desired, in view of the convergence of the perimeter, the $C^0$-bounded geometry assumption do all the work since all the notions in the definition of perimeter (\ref{Def:Perimeter}) are well transported to the limit manifolds by $C^0$-convergence of metrics. 
\end{proof}

\subsection{H{\"o}lder continuity of the multi-isoperimetric profile}

The H\"{o}lder continuity of the isoperimetric profile will be used in the next step of our framing of the proof of Theorem \ref{Theorem:Existence}.

\begin{proof}[Proof of Theorem \ref{Theorem:HolderContinuity}]
The proof goes in the same steps taken in Theorem 2 of \cite{flores2019local} with minor modifications. Let us briefly record it here, given $\epsilon > 0$, we can find $\cluster{E}$ an $N$-cluster such that $\vv{E} = \v$ and $\peri{\cluster{E}} \leq I_M(\v) + \epsilon$. Let us define $\Lambda_1$  as the set of those $h\in\{1, ..., N\}$ which $\v(h) \leq \v'(h)$. We take $p_h$ and $r_{\v'(h)}$, for $h\in\Lambda_1$, in order to have $\cluster{E}_1(h) = \cluster{E}(h)\cap\geoball{p_h}{r_{\v'(h)}}$ with $\vol{\cluster{E}_1(h)} = \v'(h)$. From the spherical Bishop-Gromov's Theorem, we obtain that

\begin{equation}\label{Eq:continuity1}
\peri{\geoball{p_1}{r_{\v'(h)}}} \leq C_1(n,k)r_{\v'(h)} \leq  C_2(n, k)\biggl( \frac{\v'(h) - \v(h)}{v_0} \biggr)^{\frac{n-1}{n}},
\end{equation}

for $h\in\Lambda_1$. Let us define $\Lambda_2$ as the complementary set of $\Lambda_1$, i.e. $\Lambda_2 = \{1, ..., N\}\setminus\Lambda_1$. Then, for $h\in\Lambda_2$, we apply Lemma 2.5 of \cite{nardulli2014generalized} for $\chamber{E}{h}$ which furnishes, for any $\v'(h)\in ]\v(h) - l, \v(h)[$, the inequality

$$ \vol{\chamber{E}{h}\cap\ball{p_h}{\biggl(\frac{\v(h) - \v'(h)}{cv_0}\biggr)^n}} \geq \min\biggl \{\v(h) - \v'(h), c\bigg( \frac{\v(h)}{I_M(\v) + \epsilon} \biggr)^n \biggr \}$$
$$ = \v(h) - \v'(h), $$

where $l = c\min\{ v_0, \biggl( \frac{\v(h)}{I_M(\v) + \epsilon} \biggr)^n \}$. We thus choose $r_{\v'(h)}$, for $h\in\Lambda_2$, such that the finite perimeter set $\cluster{E}_2(h) = \cluster{E}(h)\setminus\ball{p_h}{r_{\v'(h)}}$ satisfies $\vol{\cluster{E}_2(h)} = \v'(h)$ and again by Bishop-Gromov's Theorem we have that

\begin{equation}\label{Eq:continuity2}
\peri{\geoball{p_1}{r_{\v'(h)}}} =  C_2(n, k)\biggl( \frac{\v'(h) - \v(h)}{v_0} \biggr)^{\frac{n-1}{n}}.
\end{equation}

Finally, we define the finite perimeter sets $\cluster{E}'(h) = \cluster{E}_1(h)$, for $h\in\Lambda_1$, $\cluster{E}'(h) = \cluster{E}_2(h)$, for $h\in\Lambda_2$ and thus the cluster $\cluster{E}' = \{ \cluster{E}'(h) \}_{h=1}^{N}$ satisfies $\vv{\cluster{E}'} = \v'$. We put Proposition  \ref{EquationToPerimeter}, (\ref{Eq:continuity1}) and (\ref{Eq:continuity2}) into account to obtain

$$ I_M(\v') \leq \peri{\cluster{E}'} \leq \frac{1}{2}\biggl( \sum_{h\in\Lambda_1}\peri{\cluster{E}_1} + \sum_{h\in\Lambda_2}\peri{\cluster{E}_2} \biggr) $$
$$\leq \frac{1}{2}\sum_{h\in\Lambda_1}\biggl( \peri{\cluster{E}(h)} + \peri{\ball{p_h}{r_\v'(h)}}  \biggr) + \frac{1}{2}\sum_{h\in\Lambda_2}\biggl( \peri{\cluster{E}(h)} + \peri{\ball{p_h}{r_\v'(h)}}  \biggr)  $$
$$ \leq \peri{\cluster{E}} + C_2(n, k)\sum_{h=1}^{N}\biggl( \frac{\v'(h) - \v(h)}{v_0} \biggr)^{\frac{n-1}{n}} \leq I_M(\v) + \epsilon + C_2(n, k)\sum_{h=1}^{N}\biggl( \frac{\v'(h) - \v(h)}{v_0} \biggr)^{\frac{n-1}{n}}.$$

Letting $\epsilon$ goes to $0$ and applying the inequality above exchanging the components of $\v'$ and $\v$ as needed to reach the inequality with the modulus, we conclude the proof of the Theorem.
\end{proof}

\subsection{The structure lemma for minimizing sequences and the generalized existence theorem}

Sequences of clusters satisfying the hypothesis of Lemma \ref{Lemma:StructureLemmaForMinimizing} are called either {\bf minimizing sequences of $N$-clusters for $\v\in\mathbb{R}^{N}_{+}$} or {\bf minimizing sequences for the multi-isoperimetric problem}. In Lemma \ref{Lemma:StructureLemma} we only assumed that $(M^n, g)$ has bounded geometry and the components of the vector volumes and the perimeters of the sequence of the clusters were uniformly bounded. However, if we force the sequence of clusters to be a minimizing sequence for the multi-isoperimetric problem, we can show what is the limit of the sequence of the perimeters instead of the simply existence of it provided by item $(iii)$ of the Lemma \ref{Lemma:StructureLemma}. Moreover,  this stronger assumption on the sequence of clusters put us in position to prove that the number of pieces that we split the clusters is finite, i.e. $J\in\mathbb{N}$.

\begin{lemma}[Structure Lemma for Minimizing Sequences of Clusters]\label{Lemma:StructureLemmaForMinimizing} Assume that $(M^n, g)$ has bounded geometry. Let $\{ \cluster{E}_k \}_{k\in\mathbb{N}}$ be a sequence of $N$-clusters in $(M, g)$ with $\vv{E_\mathit{k}} = \v\in\mathbb{R}^{N}_{+}$, $\forall k\in\mathbb{N}$, and $\mathcal{P}_g(\cluster{E}_k) \to I_M(\v)$. Then there exists $J\in\mathbb{N}$ such that, for all $j\in\{1, ... J\}$, there exist a sequence of points $(p_{jk}^h)_{k\in\mathbb{N}}$, a sequence of radii $R_{jk}\to +\infty$ as $k$ goes to $+\infty$ and a volume vector $\v_j\in\mathbb{R}^{N}$, such that, if we set $\cluster{E}^{0}_k(h) =  \cluster{E}_k(h)\cap\ball{p_{0k}^h}{R_{0k}}$,

$$\cluster{E}^j_k(h) = \cluster{E}_k(h)\cap\ball{p_{jk}^h}{R_{jk}}\setminus\mathring{\bigcup}_{i=0}^{j-1}\cluster{E}_k^{i}(h),  $$

for $j\geq 1$, and the $N$-cluster $\cluster{E}'_k = \{\cup_{j=0}^J\cluster{E}_k^j(h)\}_{h=1}^N$, the following properties hold

\begin{enumerate}[$(i):$]

	\item $\v_j(h) = \lim_{k\to+\infty}\vol{\cluster{E}^j_k(h)}$, for each $h\in\{1, ..., N\}$,
		
	\item $\vv{E'_\mathit{k}} \to \sum_{j=0}^J \v_j = \v$,
	
	\item $\lim_{k\to+\infty}\mathcal{P}_g(\cluster{E}'_k) = \lim_{k\to+\infty}\mathcal{P}_g(\cluster{E}_k) = I_M(\v)$.
	
\end{enumerate}

\end{lemma}
\begin{obs}
We notice that item $(ii)$ shows that the new sequence of clusters $\cluster{E}'_k$ is not needed to be minimizing for the multi-isoperimetric problem since we just guaranteed the convergence of the vector volume sequence $\vv{E'_\mathit{k}}$ to the vector volume $\v$. 
\end{obs}
\begin{proof}
The unique part that is not a particular case of Lemma \ref{Lemma:StructureLemma} is item $(iii)$. With the aim to prove $(iii)$, we put Lemma 2.8 of \cite{nardulli2014generalized} into account (this Lemma's proof also works for the multi-isoperimetric profile) and we use the continuity of $I_M$ given by Theorem \ref{Theorem:HolderContinuity} to get that 

$$ I_M(\v) = \lim_{k\to+\infty}I_M(\vv{E'_\mathit{k}}) \leq \lim_{k\to+\infty}\peri{\cluster{E}'_k}. $$

Furthermore, item $(iii)$ of Lemma \ref{Lemma:StructureLemma} ensures the reverse inequality. To see that $J<+\infty$, we proceed by contradiction. As in the second proof of Theorem 3 in \cite{nardulli2014generalized}, we can prove the existence of a constant $M$, which does not depend on $J$, such that

\begin{equation}\label{Jfinite1}
M \geq \frac{I_M(\v'_J)}{\norm{\v'_J}}, \ \ \v'_J = \v - \sum_{j=0}^{J}\v_j,
\end{equation}

whenever $J$ is sufficiently large and thus $\v'_J$ has small norm. Now, we want to obtain a lower bound of $I_M(\v'_J)$, possibly depending on $\v'_J$, in order to lead (\ref{Jfinite1}) into a contradiction. By Proposition \ref{EquationToPerimeter}, Caccioppoli sets operations and putting into account the density properties of the interfaces (precisely, equation (\ref{normalinterfaces})), we have

$$ I_M(\v'_J) \geq \inf\left \{ \frac{1}{2}\sum_{h=1}^{N}\peri{\chamber{E}{h}} : \cluster{E} \ \text{is a N-cluster with} \ \vv{E}=\v'_J \right \} \geq  $$
$$ \frac{1}{2}\inf\left \{ \per{\chamber{E}{1}}{\chamber{E}{2}^{(0)}} + \per{\chamber{E}{2}}{\chamber{E}{1}^{(0)}} + \sum_{h=3}^{N}\peri{\chamber{E}{h}} : \cluster{E} \ \text{is a N-cluster with} \ \vv{E}=\v'_J \right \}$$
$$ = \frac{1}{2}\inf\left \{ \peri{\chamber{E}{1}\cup\chamber{E}{2}} + \sum_{h=3}^{N}\peri{\chamber{E}{h}} : \cluster{E} \ \text{is a N-cluster with} \ \vv{E}=\v'_J \right \},$$

Repeating the same procedure, we come up with

$$ I_M(\v'_J) = \frac{1}{2}\inf\left \{ \peri{\bigcup_{h=1}^{N}\chamber{E}{h}} : \cluster{E} \ \text{is a N-cluster with} \ \vv{E}=\v'_J \right \} \geq^{\ast} $$
$$ \frac{1}{2C}\inf\left \{ \vol{\bigcup_{h=1}^{N}\chamber{E}{h}}^{\frac{n-1}{n}} : \cluster{E} \ \text{is a N-cluster with} \ \vv{E}=\v'_J \right \} = $$
$$ \frac{1}{2C}\inf\left \{ \left(\sum_{h=1}^{N}\vol{\chamber{E}{h}}\right)^{\frac{n-1}{n}} : \cluster{E} \ \text{is a N-cluster with} \ \vv{E}=\v'_J \right \} =$$
$$ \frac{1}{2C}\left( \sum_{h=1}^{N}\v'_J(h) \right)^{\frac{n-1}{n}},$$

where in $(\ast)$ we take $J$ big enough to apply the Isoperimetric Inequality (Proposition \ref{Lemma:IsoperimetricInequality} below) for small volume, i.e. $\norm{\v'_J}$ sufficiently small. The last chain of inequalities and (\ref{Jfinite1}) provide

\begin{equation}\label{Jinfinity2}
2MC \geq \frac{I_M(\v'_J)}{\norm{\v'_J}} \geq \frac{\left( \sum_{h=1}^{N}\v'_J(h) \right)^{\frac{n-1}{n}}}{\norm{\v'_J}} \geq^{\ast} \frac{K(n, N)\norm{\v'_J}^{\frac{n-1}{n}}}{\norm{\v'_J}} = \frac{K(n, N)}{\norm{\v'_J}^\frac{1}{n}} \longrightarrow
 +\infty,
\end{equation}

as $J$ goes to $+\infty$, where in $(\ast)$ the constant $K(n, N)$, which does not depend on $J$, appears from the equivalence of all norms in Euclidean spaces (specifically $\mathbb{R}^{N}$). Equation (\ref{Jinfinity2}) states the contradiction that we needed. So, $J$ has to be finite.
\end{proof}

Finally, we are in position to prove Theorem \ref{Theorem:Existence} which states the existence of minimizing clusters for the multi-isoperimetric problem. 

\begin{proof}[Proof of Theorem \ref{Theorem:Existence}]
The proof turns out to be a simple application of Theorem \ref{Theorem:Compactness} and Lemma \ref{Lemma:StructureLemmaForMinimizing}.
\end{proof}

\subsection{Boundedness of isoperimetric clusters}

We recall the statement of the classical isoperimetric inequality in the Riemannian setting which can be consulted in Lemma 3.2 of Emmanuel Hebey work (\cite{EB}), it requires that $E$ is open and has smooth boundary. Nevertheless, we state the inequality for finite perimeter sets.

\begin{lemma}[Isoperimetric Inequality]\label{Lemma:IsoperimetricInequality} Let $(M^n, g)$ be a complete Riemannian manifold with bounded geometry. There exist positive constants $C$ and $\eta$ depending only on $n, k$ and $v_0$ (Definition \ref{Def:BoundedGeometry}) such that, for any finite perimeter set $E$ of $(M^n, g)$ with $\vol{E}\leq \eta$, we have

$$\vol{E}^{\frac{n-1}{n}} \leq C\peri{E}.$$

\end{lemma}
\begin{proof}
The proof of this inequality is a standard approximation argument using Lemma 3.2 of \cite{EB} and Lemma 2.4 of \cite{flores2014continuity}.
\end{proof}

The proof of the boundedness of isoperimetric clusters goes back to ideas given by Frank Morgan on Chapter 13 of \cite{MorGMT} in the Euclidean setting. We adapted these ideas to the context of Riemannian manifolds with bounded geometry.

\begin{proof}[Proof of Theorem \ref{Prop:BoundednessClusters}]
Let $\cluster{E}$ be an isoperimetric $N$-cluster of $(M^n, g)$ and fix $p\in M$. Let us define the function $V:(0, +\infty ) \to (0, +\infty )$ which measures the amount cluster's volume outside of large balls as follows 

$$ V(r) = \sum_{h=1}^{N}\vol{\chamber{E}{h}\setminus\ball{p}{r}}.$$

By the coarea formula, we get that

\begin{equation}\label{Eq:boundedness1}
V '(r) = -\sum_{h=1}^{N}\per{\ball{p}{r}}{\chamber{E}{h}}.
\end{equation}

If we set $A_h(r) = \per{\chamber{E}{h}}{M\setminus\ball{p}{r}}$, equation (\ref{Eq:boundedness1}), standard arguments with Caccioppoli sets and operations with the perimeter measure yield

\begin{equation}\label{Eq:boundedness2}
\modulo{V '(r)} + \sum_{h=1}^{N}A_h(r) = \sum_{h=1}^{N}\peri{\chamber{E}{h}\cap\left(M\setminus\ball{p}{r}\right)},
\end{equation}  

for almost all $r$ big enough. From standard Riemannian comparison geometry techniques, we can easily see that $V(r)$ is decreasing and tends to $0$ as $r$ goes to $+\infty$. Then, for almost all $r$ sufficiently large, we can apply Lemma \ref{Lemma:IsoperimetricInequality} and (\ref{Eq:boundedness2}) to obtain

\begin{equation}\label{Eq:Boundedness3}
\begin{split}
\modulo{V '(r)} + \sum_{h=1}^{N}A_h(r) \geq C\vol{\left(\bigcup_{h=1}^{N}\chamber{E}{h}\right)\cap\left(M\setminus\ball{p}{r}\right)}^{\frac{n-1}{n}} =^{\ast} \\
C\left(\sum_{h=1}^{N}\vol{\chamber{E}{h}\cap\left(M\setminus\ball{p}{r}\right)}\right)^{\frac{n-1}{n}} = CV(r)^{\frac{n-1}{n}}.
\end{split}
\end{equation}

where, by the definition of clusters (Definition \ref{Def:Clusters}), $(\ast)$ follows from the fact that the chambers do not overlap in the measure-theoretic sense (i.e. the intersection has volume zero). By an application of Lemma 13.5 of \cite{MorGMT} and Proposition \ref{EquationToPerimeter}, we have

\begin{equation}\label{Eq:Boundedness4}
\modulo{V'(r)}+ cV(r) \geq \per{\cluster{E}}{M\setminus\ball{p}{r}} = \frac{1}{2}\sum_{h=1}^{N}A_h(r).
\end{equation}

Adding inequalities (\ref{Eq:Boundedness3}) and (\ref{Eq:Boundedness4}) furnishes

\begin{equation}\label{Eq:Boundedness5}
3\modulo{V'(r)} \geq  CV(r)^{\frac{n-1}{n}} - 2cV(r) \geq \frac{C}{4}V(r)^{\frac{n-1}{n}}, 
\end{equation}

where $V$ being decreasing ensures the last inequality. Finally, if we suppose that one of the chambers of $\cluster{E}$ is unbounded, we obtain that $V(r)>0$ for all $r$ and hence, by (\ref{Eq:Boundedness5}), we get

$$ \left( V^{\frac{1}{n}} \right)' = \frac{V'}{nV^{\frac{n-1}{n}}} \leq -\frac{C}{12n} < 0. $$

Finally, by the mean value theorem and since $\frac{C}{12n}$ is negative and independent of $r$, the last equation contradicts that $V$ is decreasing and positive. 
\end{proof}

\newpage
\section{Application to the classical existence of isoperimetric clusters}\label{application-classical-results}

To this aim, let us contextualize the classical setting that we mentioned before.

\begin{defi}\label{Def:ConvergeToSpaceForm}
We say that $(M^n, g)$ is $C^0$-locally asymptotically a space form, if it has $C^0$-bounded geometry and for every diverging sequence of points $(p_k)$ we have 

$$ (M, g, p_k) \to (\mathbb{M}_k^n, g_{standard}, x)$$

in the $C^0$-topology, where $\mathbb{M}_k^n$ is a $n$-dimensional space form of curvature $k$ and $x$ is any point in $\mathbb{M}_k^n$.
\end{defi}

\begin{proof}[Proof of Theorem \ref{Theorem:ClassicalExistence}]
By Lemma \ref{Lemma:StructureLemmaForMinimizing}, Theorem \ref{Theorem:Existence} and Definition \ref{Def:ConvergeToSpaceForm}, we get that at most one of the limit chamber's pieces $\cluster{E}^{j}_k(h)$ lives at the limit manifold $\mathbb{M}_k^n$. If all limit chamber's pieces live in $M$, we have nothing to do and the proof is done. However, if one of the limit chamber's pieces lives at $\mathbb{M}_k^n$, we start recalling that 

$$(\cluster{E}_{\infty}(h), g, p_{j\infty}^{h}) = \bigcup_{j=0}^{J}\lim_{k\to+\infty}(\cluster{E}^{j}_k(h), g, p_{jk}^{h}),$$

where the limit is taken in the $C^0$-topology. Set $\lim_{k\to+\infty}(\cluster{E}^{j}_k(h), g, p_{jk}^{h}) = (\cluster{E}^{j}_{\infty}(h), g, p_{j\infty}^{h})$ and suppose that the $\cluster{E}^{J}_{\infty}(h)$ is the one who lives at $\mathbb{M}_{k}^{n}$, for all $h\in\{1,\cdots ,N\}$. Denote the volume of $\cluster{E}^{J}_{\infty}(h)$ by $\v^{J}(h)$. Then, we choose a metric ball $\ball{p_h}{r_h}$ in $M$ with volume $\v^{J}(h)$ and positive distance from $\cup_{j=0}^{J-1}\cluster{E}^{j}_{\infty}(h)$, what we can done thanks to the boundedness of the isoperimetric clusters (Proposition \ref{Prop:BoundednessClusters}). We have to introduce two new $N$-cluster to simplify further equations, even though they will be used only for this proof, define

$$ \cluster{A} = \{ \mathring{\bigcup}_{j=0}^{J-1}\cluster{E}^{j}_{\infty}(h) \}_{h=1}^{N}, $$
$$ \cluster{B} = \{ \cluster{E}^{J}_{\infty}(h) \}_{h=1}^{N}. $$

Recalling that balls are the isoperimetric regions in space forms, therefore we get that

$$ I_M(\v) = \peri{\cluster{E}_{\infty}} = \peri{\cluster{A}} + \peri{\cluster{B}} = \peri{\cluster{A}} + \mathcal{P}_{g_{standard}} (\cluster{B}) \geq $$
$$ \peri{\cluster{A}} + \frac{1}{2}\sum_{h=1}^{N}\tilde{I}_{\mathbb{M}_{k}^{n}}\left(\v^{J}(h)\right)  = \peri{\cluster{A}} + \frac{1}{2}\sum_{h=1}^{N}\mathcal{P}_{g_{standard}}\left(\textbf{B}_{M_{k}^n}(\v^{J}(h))\right )$$
$$\geq^{\ast}  \peri{\cluster{A}} + \frac{1}{2}\sum_{h=1}^{N}\peri{\ball{p_h}{r_h}} = \frac{1}{2}\sum_{h=1}^{N}\peri{\ball{p_h}{r_h}\mathring{\cup}\left(\mathring{\bigcup}_{j=1}^{J-1}\cluster{E}^{j}_{\infty}(h)\right)},$$

where $\tilde{I}_{\mathbb{M}_{k}^{n}}$ denotes the isoperimetric profile of the space form $\mathbb{M}_{k}^{n}$, $\textbf{B}_{M_{k}^n}(\v^{J}(h))$ denotes a ball of volume $\v^{J}(h)$ in $\mathbb{M}_{k}^{n}$, $(\ast)$ is due to equation (2) of the proof of Proposition 3.2 in \cite{mondino2012existence} and Proposition \ref{EquationToPerimeter}. Finally, we set the $N$-cluster $\cluster{E}$ as follows

$$ \cluster{E} = \left \{ \ball{p_h}{r_h}\mathring{\cup}\left(\mathring{\bigcup}_{j=0}^{J-1}\cluster{E}^{j}_{\infty}(h)\right) \right \}_{h=1}^{N} $$
\end{proof}

\newpage
\section{Appendix}\label{appendix}

In this section, we furnish theorems which show how exchanging volume process and volume fixing variations work in Riemannian manifolds with bounded geometry. The reason to write this appendix in this work is that these techniques are vastly used in the regularity theory for clusters in Euclidean spaces. So, it can be useful to further developments of regularity theory for clusters in Riemannian manifolds with bounded geometry and, possibly, more generally, in doubling metric spaces.	

We will fix the following notations for the different types of support $\supp f = \{x\in dom f: f(x) \neq 0\}$ and $\spt f = \{ x\in dom f : f(x)\neq x\}$. We will denote the geodesic balls with the subscript $M$ and the metric balls with the subscript $g$, i.e. $\geoball{x}{r}$ and $\ball{x}{r}$ respectively.

\begin{lemma}[Infinitesimal volume exchange]\label{Lemma:Infinitesimal}
Suppose that $(M^n, g)$ has bounded geometry. Let $\cluster{E}$ be an $N$-cluster in $(M^n, g)$, $y\in\intface{E}{h}{k}$, $ 0 \leq h < k \leq N$ and $\delta >0$. Then there exist $\epsilon_1 \doteq \epsilon_1(\cluster{E}, y, \delta) < inj_M, \epsilon_2 \doteq \epsilon_2(\epsilon_1, n), C_0 \doteq C_0(n, \epsilon_1)$ and a one-paramater family of diffeomorphims $\{\famdif{\phi}\}_{\modulo{t}<\epsilon_1}$ depending on $y$ with for all $\modulo{t}<\epsilon_1$
\begin{equation}\label{eq:infinitesimal0}
\spt\famdif{\phi} \cc \geoball{y}{\epsilon_1} 
\end{equation}
and satisfying the following properties:

\begin{enumerate}[(i):]

\item If $\cluster{E'}$ is an N-cluster, $d_{\mathcal{F}, g}(\cluster{E}, \cluster{E'}) <\epsilon_2$ and $\modulo{t}<\epsilon_1$, then
	\begin{equation}\label{eq:infinitesimal1}
  	 	\bigg\vert  \frac{d}{dt}\vol{\famdif{\phi}(\chamber{E}{h}\cap\geoball{y}{\epsilon_1})} - 1 \bigg\vert < \delta,
    \end{equation}	  
	\begin{equation}\label{eq:infinitesimal2}
  	 	\bigg\vert  \frac{d}{dt}\vol{\famdif{\phi}(\chamber{E}{k}\cap\geoball{y}{\epsilon_1})} + 1 \bigg\vert < \delta,
    \end{equation}	  
    	\begin{equation}\label{eq:infinitesimal3}
  	 	\bigg\vert  \frac{d}{dt}\vol{\famdif{\phi}(\chamber{E}{i}\cap\geoball{y}{\epsilon_1})} \bigg\vert < \delta \ \ \ \ \, i \neq h, k,
    \end{equation}	  
	\begin{equation}\label{eq:infinitesimal4}
  	 	\bigg\vert  \frac{d^2}{dt^2}\vol{\famdif{\phi}(\chamber{E}{i}\cap\geoball{y}{\epsilon_1})} \bigg\vert < C_0 \ \ \ \ \, 1 \leq i \leq N.
    \end{equation}	  
    
\item If $\Sigma$ is an $\hm$-rectifiable set in $(M, g)$ and $\modulo{t}<\epsilon_1$, then
	
	\begin{equation}\label{eq:infinitesimal5}
		\bigg\vert \hmde{\famdif{\phi}(\Sigma)} - \hmde{\Sigma} \bigg\vert \leq C_0 \hmde{\Sigma}\modulo{t}.
	\end{equation}

\end{enumerate}
\end{lemma}
\begin{proof}
First of all we will construct a vector field in $\R^n$ which modifies sets inside a small ball. To this end, we choose a function $v\in C^{\infty}_c(\textbf{B}_{\euc{n}}(0, 1))$ such that, for all $\epsilon > 0$ if we define $v_\epsilon(x) = \epsilon^{1-n}v(x/\epsilon)$, then we have $v_\epsilon \in C^{\infty}_c(\textbf{B}_{\euc{n}}(0, \epsilon))$ and $\norm{v_\epsilon}_{\infty} \leq \epsilon^{-n}C(n)$. Let us denote by $Q_\nu$ the orthogonal transformation which carries $\nu\in\sphere{n-1}_{\R^n}$ to $e_n$ (n-th vector of $\R^n$'s basis). Then, we can define $T\doteq T[\epsilon, \nu]\in C^{\infty}_{c}(\textbf{B}_{\euc{n}}(0, \epsilon), \R^{n})$ as follows

$$ T(x) = v_\epsilon(Q_\nu(x))\nu $$.

It is straighforward to see that this notions can be extended for any vector space, let us now apply this construction to $M$. For any $0<\epsilon < inj_M, y\in M$ and $\nu\in\sphere{n-1}_{T_{y}M}$, we fix $T \doteq T[y, \epsilon, \nu]$ and define $\xi \doteq \xi[y, \epsilon, \nu]\in\mathfrak{X}^{\infty}_{c}(M)$ with $\supp\xi\cc\geoball{y}{\epsilon}$ as follows

$$ \xi(p) = i_y^p \biggl( T[y, \epsilon, \nu]\left(exp^{-1}_{y}(p)\right) \biggr), $$

where $i_y^p : T_yM \to T_pM$ is any linear isometry. We now take the flow of $\xi$ to be the one paramater family of diffeomorphims $\{\famdif{\phi}\}_{\modulo{t}<\epsilon}$, it is clear that $\spt\famdif{\phi}\cc\geoball{y}{\epsilon}$. Let us start verifying $(ii)$, by the area formula for rectifiable sets we have that for all $\Sigma\subset M$ $\hm$-rectifiable

\begin{equation}\label{infinitesimal1}
\hmde{\famdif{\phi}(\Sigma)} = \int_{\Sigma}J^{\Sigma}\famdif{\phi}\diff\hm
\end{equation}

for any $\modulo{t} < \epsilon < inj_M$. Since $\varphi_{0} = id$ and $(t, p) \mapsto \famdif{\phi}$ is smooth in $(-\epsilon, \epsilon)\times M$, we ensure the existence of  $\epsilon_1 \doteq \epsilon_1(n, y, \epsilon)$, $C_0' \doteq C_0'(n, y, \epsilon)$ such that, for $\epsilon$ sufficiently small,

\begin{equation}\label{infinitesimal2}
\bigg\vert J^{\Sigma}\famdif{\phi} - 1\bigg\vert\leq C_0'\modulo{t},
\end{equation}

$\hm$-a.e. in $\Sigma$ and thus $(ii)$ follows directly from (\ref{infinitesimal1}) and (\ref{infinitesimal2}). By (\ref{eq:infinitesimal0}), we get $C' \doteq C'(n, y, \epsilon)$ such that

\begin{equation}\label{infinitesimal0}
\sup_{M}\biggl( \frac{d}{dt}J\famdif{\phi} + \frac{d^2}{dt^2}J\famdif{\phi}\biggr) \leq C'.
\end{equation}

Applying the area formula for $\chamber{E}{i}, 0\leq i\leq N,$ and (\ref{infinitesimal0}), we obtain

$$ \bigg\vert  \frac{d^2}{dt^2}\vol{\famdif{\phi}(\chamber{E}{i}\cap\geoball{y}{\epsilon})} \bigg\vert \leq C'\vol{\geoball{y}{\epsilon}}, $$

what conclude the proof of (\ref{eq:infinitesimal4}). It is straightforward calcultion to see that it suffices to prove (\ref{eq:infinitesimal1}), (\ref{eq:infinitesimal2}) and (\ref{eq:infinitesimal3}) for $\cluster{E'} = \cluster{E}$ and $t=0$. By the first variation of volume (one can consult Proposition 17.8 in \cite{MaggiBook2012} and adapt the proof for the context of Riemannian manifolds with bounded geometry), we obtain that

\begin{equation}\label{infinitesimal3}
\frac{d}{dt}\bigg\vert_{t=0}\vol{\famdif{\phi}(E\cap\geoball{y}{\epsilon})} = \int_{\geoball{y}{\epsilon}\cap \frored{E}}\prodin{\xi}{\nu_E}\diff\hm
\end{equation}

for any finite perimeter set $E$. So, let us use (\ref{infinitesimal3}) with $E=\chamber{E}{i}$, $\xi = \xi[y, \epsilon, \nu_{\chamber{E}{h}}(y)]$ and $i_y^p$ such that $i_y^p(\nu_{\chamber{E}{h}}(y)) = \nu_{\chamber{E}{h}}(p)$, then it turns 

\begin{equation}\label{infinitesimal4}
\frac{d}{dt}\bigg\vert_{t=0}\vol{\famdif{\phi}(\chamber{E}{i}\cap\geoball{y}{\epsilon})} = \int_{\geoball{y}{\epsilon}\cap\frored{\chamber{E}{i}}}\epsilon^{1-n}v\biggl(\frac{1}{\epsilon}Q_{\nu_{\chamber{E}{h}}(y)}(exp^{-1}_{y}(p))\biggr)\prodin{\nu_{\chamber{E}{h}}(p)}{\nu_{\chamber{E}{i}}(p)}\diff\hm(p).
\end{equation}

Taking $i \neq h, k$, we immediately have that

\begin{equation}\label{infinitesimal5}
\frac{d}{dt}\bigg\vert_{t=0}\vol{\famdif{\phi}(\chamber{E}{i}\cap\geoball{y}{\epsilon})} \leq \left(\sup_{\geoball{y}{\epsilon}}\modulo{v} \right)\frac{\per{\chamber{E}{i}}{\geoball{y}{\epsilon}}}{\epsilon^{n-1}}.
\end{equation}

From (\ref{density0'}) follows (\ref{eq:infinitesimal3}) for any $\epsilon$ sufficiently small depending on $\cluster{E}, y$ and $\delta$. Now, we set $ i = h $ and then (\ref{infinitesimal3}) furnishes 

\begin{equation}\label{infinitesimal16}
\frac{d}{dt}\bigg\vert_{t=0}\vol{\famdif{\phi}(\chamber{E}{i}\cap\geoball{y}{\epsilon})} = \int_{\geoball{y}{\epsilon}\cap\frored{\chamber{E}{h}}}v_\epsilon\biggl(Q_{\nu_{\chamber{E}{h}}(y)}(exp^{-1}_{y}(p))\biggr)\diff\hm(p).
\end{equation}

Analogously, using (\ref{normalinterfaces}) we get

\begin{equation}\label{infinitesimal17}
\frac{d}{dt}\bigg\vert_{t=0}\vol{\famdif{\phi}(\chamber{E}{i}\cap\geoball{y}{\epsilon})} = -\int_{\geoball{y}{\epsilon}\cap\frored{\chamber{E}{k}}}v_\epsilon\biggl(Q_{\nu_{\chamber{E}{h}}(y)}(exp^{-1}_{y}(p))\biggr)\diff\hm(p).
\end{equation}

In order to obtain (\ref{eq:infinitesimal1}) and (\ref{eq:infinitesimal2}) from (\ref{infinitesimal16}) and (\ref{infinitesimal17}), it remains to choose a function $v$ at the beginning also satisfying $\lim_{\epsilon\to 0}v_\epsilon (0) = 1$.

\end{proof}

For the next result, we will use the notation

$$ V = \biggl \{ \v \in \R^{N+1} : \sum_{h=0}^{N}\v(h) = 0 \biggr \} \subset \R^{N+1}.$$

\begin{Res}[Volume fixing variation]\label{Theorem:VolumeFixingVariation}
Suppose that $(M^n, g)$ has bounded geometry. For all $N$-cluster $\cluster{E}$ in $(M^n, g)$, $y\in\intface{E}{h}{k}$, there exist positive constants $\eta, \epsilon_1, \epsilon_2, C_1$ and $R$ such that for all $N$-cluster $\cluster{E'}$ with $d_{\mathcal{F}, g}(\cluster{E}, \cluster{E'})<\epsilon_2$ there exists a function,

$$ \Psi: \biggl( (-\eta, \eta)^{N+1}\cap V\biggr)\times M \to M, $$

of class $C^1$ with the following properties for $\v \in (-\eta, \eta)^{N+1}\cap V$ :

\begin{enumerate}[(i):]

	\item $\Psi(\v, \cdot)$ is a diffeomorphism with relatively compact support,
	
	\item it holds 
	
	$$ \vol{\Psi (\v, \chamber{E'}{h})\cap\spt \Psi} = \vol{\chamber{E'}{h}\cap\spt\Psi} + \v(h) \ \ \ \, 0 \leq h \leq N,$$

	\item $\Sigma$ is a $\hm$-rectifiable set in $(M, g)$, then
	
	$$ \biggl | \hmde{ \Psi(\v, \Sigma )} - \hmde{\Sigma}\biggr | \leq C_1 \hmde{\Sigma}\sum_{h=0}^{N}\modulo{\v(h)}, $$
	
	\item there exist $M\in [N, 2N^2]\cap\mathbb{N}$ and a family $\{y_{\alpha}\}_{\alpha=1}^{M}$ of interface points of $\cluster{E}$ with $\modulo{y_\alpha - y_\beta} > 4\epsilon_1$ for $1\leq\alpha <\beta\leq M$ such that 
	
	$$spt\Psi (\v, \cdot ) \subset \bigcup_{\alpha=1}^{M}\geoball{y_\alpha}{\epsilon_1}, $$
	
	\item if $\{ z_{\alpha} \}_{\alpha=1}^{M}$ is another family of interface points of $\cluster{E}$ with $\modulo{z_\alpha - z_\beta} > 4\epsilon_1$ for $1\leq \alpha < \beta \leq M$ and $y_\alpha$ and $z_\alpha$ belongs to the same interface for $1\leq \alpha \leq M$, then there exists $\eta ', \epsilon_1 ', \epsilon_2 ', C_1 '$ and $R'$ such that the conclusions of this theorem is validated for the new constants and $z_\alpha$.

\end{enumerate}
 
\end{Res}
\begin{proof}
The proof goes mutatis mutandis as in the proof of Theorem 29.14 in \cite{MaggiBook2012} using the preceding result (Lemma \ref{Lemma:Infinitesimal}).
\end{proof}

      \newpage
      \markboth{References}{References}
      \bibliographystyle{alpha}
      \bibliography{OptimalClusters}

\end{document}